\documentclass[a4paper]{amsart}
\usepackage[latin1]{inputenc}
\usepackage{amssymb}
\usepackage{amsmath}
\usepackage{mathrsfs}
\usepackage{eufrak}
\usepackage{amsthm}
\usepackage{amsfonts}
\usepackage{textcomp}
\usepackage{graphicx}
\usepackage[pdftex]{color}
\usepackage{paralist}
\usepackage[shortlabels]{enumitem}
\usepackage{hyperref}
\usepackage{comment}
\usepackage[arrow, matrix, curve]{xy}

\newtheorem*{corollary*}{Corollary}
\newtheorem*{Result 1}{Result 1}
\newtheorem*{Result 2}{Result 2}
\newtheorem*{Result 3}{Result 3}
\newtheorem*{Result 4}{Result 4}
\newtheorem*{Case 2}{Case 2}
\newtheorem*{Case 1}{Case 1}
\newtheorem*{Theorem(Yamagata)}{Theorem(Yamagata)}
\newtheorem*{Mainlemma}{Mainlemma}
\newtheorem*{Yamagata conjecture}{Yamagata conjecture}
\newtheorem*{Nakayama conjecture}{Nakayama conjecture}
\newtheorem*{Theorem of Mueller}{Theorem of Mueller}
\newtheorem*{Minimal projective resolutions and minimal injective resolutions over Nakayama algebras}{Minimal projective resolutions and minimal injective resolutions over Nakayama algebras}
\newtheorem*{Length of the indecomposable left modules}{Length of the indecomposable left modules}
\newtheorem*{Structure of indecomposable injective modules}{Structure of indecomposable injective modules}
\newtheorem*{Problem}{Problem}
\newtheorem{theorem}{Theorem}[section]
\newtheorem*{Conjecture (Abrar)}{Conjecture (Abrar)}

\newtheorem{corollary}[theorem]{Corollary}
\newtheorem{lemma}[theorem]{Lemma}
\newtheorem{proposition}[theorem]{Proposition}

\newtheorem*{claim*}{Claim}

\theoremstyle{definition}
\newtheorem{definition}[theorem]{Definition}

\newtheorem{example}[theorem]{Example}

\theoremstyle{remark}

\numberwithin{equation}{theorem}

\makeatletter
\renewcommand*\env@matrix[1][\
arraystretch]{%
  \edef\arraystretch{#1}%
  \hskip -\arraycolsep
  \let\@ifnextchar\new@ifnextchar
  \array{*\c@MaxMatrixCols c}}
\makeatother

\begin{document}

\title{Upper bounds for the dominant dimension of Nakayama and related algebras}
\date{\today}



\author{Ren\'{e} Marczinzik}
\address{Institute of algebra and number theory, University of Stuttgart, Pfaffenwaldring 57, 70569 Stuttgart, Germany}
\email{marczire@mathematik.uni-stuttgart.de}

\begin{abstract}
Optimal upper bounds are provided for the dominant dimensions of Nakayama algebras and more generally algebras $A$ with an idempotent $e$ such that there is a  minimal faithful injective-projective module $eA$ and such that $eAe$ is a Nakayama algebra. This answers a question of Abrar and proves a conjecture of Yamagata for monomial algebras.
\end{abstract}

\maketitle
\section*{Introduction}
The dominant dimension domdim($A$) of a finite dimensional algebra $A$ is defined as follows:
Let \newline \centerline{$0 \rightarrow A \rightarrow I_0 \rightarrow I_1 \rightarrow I_2 \rightarrow ...$} \newline be a minimal injective resolution of the right regular module $A$. If $I_0$ is not projective, we set domdim($A$):=0, and otherwise \newline
\centerline{domdim($A$):=$\sup \{ n | I_i $ is projective for $i=0,1,...,n \}$+1.} \newline
Note that the dominant dimension is invariant under Morita equivalence and also under field extensions (see \cite{Mue} Lemma 5). Thus we can assume from now on that all algebras are basic and split over the field unless stated otherwise. For this reason we assume throughout this article that algebras are given by quiver and relations if not stated otherwise.
One of the most famous conjectures in the representation theory of finite dimensional algebras is the Nakayama conjecture. This conjecture states that the dominant dimension of a non-selfinjective finite dimensional algebra is always finite (see \cite{Nak}). A stronger conjecture was given in \cite{Yam}, where Yamagata conjectured that the dominant dimension is bounded by a function depending on the number of simple modules of a non-selfinjective algebra.
Since the finiteness of the dominant dimension of a non-selfinjective algebra follows as a corollary of the finiteness of the finitistic dimension, the Nakayama conjecture is true for many classes of algebras. Examples include algebras with representation dimension smaller than or equal to 3 (see \cite{IgTo}). 
In contrast to that, explicit optimal bounds or values for the dominant dimension are rarely known for given classes of algebras. Here and in the following, an optimal bound denotes a bound on the dominant dimension such that the value of this bound is also attained in the given class of algebras. This leads to the following problem:

\begin{Problem}
For a given class of connected non-selfinjective algebras, find optimal bounds for the dominant dimensions.
\end{Problem}
In \cite{Abr} Theorem 1.2.3., Abrar shows that the dominant dimension of connected quiver algebras with an acyclic quiver is bound by the number of projective-injective indecomposable modules and that this bound is optimal for this class of algebras.
Recall that Nakayama algebras are defined as algebras having the property that every indecomposable projective left or right module is uniserial, see for example \cite{SkoYam} chapter I.10. for more information on Nakayama algebras.
One conjecture about the optimal bound of the dominant dimension for non-selfinjective Nakayama algebras was given by Abrar in \cite{Abr} as Conjecture 4.3.21 :
\begin{Conjecture (Abrar)}
Let $A$ be a non-selfinjective Nakayama algebra with $n \geq 3$ simple modules. Then 
$$\text{domdim}(A) \leq 2n-3.$$
\end{Conjecture (Abrar)}
In \cite{Abr}, Abrar calculated the dominant dimension for many Nakayama algebras and there the biggest value attained by a non-selfinjective Nakayama algebra with $n$ simple modules was $2n-3$, which lead him to his conjecture. Another article where the dominant dimension of a large class of Nakayama algebras is calculated is \cite{CIM}. \newline
We have four main results. The structure of the results is as follows: result 2 corrects and proves the conjecture of Abrar in a more general setting. Result 2 is a consequence of the much more general result 1, which also proves Yamagata's conjecture for the class of algebras given in result 1.
Our methods also give a bound on the finitistic dominant dimension, defined below, which is result 3.
Result 4 gives an explicit formula for the dominant dimension for Nakayama algebras that are Morita algebras in the sense of Kerner and Yamagata (see \cite{KerYam}). Result 4 is used to show that the optimal bound $2n-2$ is attained at a Nakayama algebra with $n$ simple modules.
\begin{Result 1}
(see \hyperref[best result]{ \ref*{best result}})
Let $A$ be a finite dimensional non-selfinjective algebra with dominant dimension at least 1 and minimal faithful injective-projective module $eA$. Let $s$ be the number of nonisomorphic indecomposable injective-projective
modules in mod-$A$ and assume that $eAe$ is a Nakayama algebra. Then the dominant dimension of $A$ is bounded by $2s$.
\end{Result 1}
Recall that a monomial algebra is a quiver algebra $KQ/I$ with admissible ideal $I$ generated by zero reations. Any Nakayama algebra is a monomial algebra.
We will prove that for every monomial algebra $A$ and every idempotent $e \in A$ with $eA$ being minimal faithful projective-injective, $eAe$ is a Nakayama algebra. So we can apply result 1 and we get the following (generalised) answer to the conjecture of Abrar:
\begin{Result 2}
(see \ref{Grenze2new} and \hyperref[strictstrict]{ \ref*{strictstrict}})
Let $A$ be a non-selfinjective monomial algebra with $n \geq 2$ simple modules.
Then domdim($A$) $\leq 2n-2$ and the bound is optimal and attained at a Nakayama algebra.
\end{Result 2}

We also introduce the finitistic dominant dimension fdomdim($A$) of an algebra $A$, which is defined as the supremum of all dominant dimensions of all modules having finite dominant dimension.
We prove the following for Nakayama algebras, which also gives an alternative proof of Abrar's conjecture:
\begin{Result 3}
(see \hyperref[findom]{ \ref*{findom}})
Let $A$ be a non-selfinjective Nakayama algebra with $n \geq 2$ simple modules. Then fdomdim($A) \leq 2n-2$.
\end{Result 3}
We remark that in general the finitistic dominant dimension is larger than the dominant dimension, see the example in \ref{examplefindomdim} and \cite{Mar} for more on the finitstic dominant dimension.
In the last section an explicit formula is given for the dominant dimensions of Nakayama algebras that are also Morita algebras as defined in \cite{KerYam}. In the case of Nakayama algebras that are also gendo-symmetric algebras (defined in \cite{FanKoe}) the dominant and Gorenstein dimensions have a surprising graph theoretical interpretation. This is used to construct a gendo-symmetric Nakayama algebra with $n \geq 2$ simple modules and with dominant dimension equal to $2n-2$. We give here the result for gendo-symmetric Nakayama algebras, and refer to section 3 of this paper for the general case and details.
In the following $\equiv_n$ denotes equality mod $n$.
\begin{Result 4}
(see \hyperref[formgendo]{ \ref*{formgendo}} and \hyperref[gordimgen]{ \ref*{gordimgen}})
Let $A$ be a symmetric Nakayama algebra with Loewy length $w \equiv_n 1$ and $n$ simple modules. Let
$M=\bigoplus\limits_{i=0}^{n-1}{e_i A} \oplus \bigoplus\limits_{i=1}^{r}{e_{x_i} A / e_{x_i} J^{w-1}}$ with the $x_i$ pairwise different for all $i \in \{ 1,...,r \}$. The $x_i$ in the quiver of $A$ are called special points. Then $B:=End_A(M)$ is a Nakayama algebra and the following holds:
\begin{align*}
\text{domdim}(B)&= 2\inf \{ s \geq 1 \mid \exists i,j: x_i + s  \equiv_n x_j  \}
\end{align*}
So the dominant dimension is just twice the (directed) graph theoretical minimal distance of two special points which appear in $M$.
Furthermore $B$ has Gorenstein dimension
$$2\sup \{ u_i \mid u_i=\inf \{b \geq 1 \mid \exists j:\ x_i+b \equiv_n x_j  \} \},$$
which is twice the maximal distance between two special points.
\end{Result 4}

In forthcoming work we will also give formulae to calculate the finitistic dominant dimension of Nakayama algebras that are Morita algebras. There the finitistic dominant dimension of non-selfinjective gendo-symmetric Nakayama algebras will be shown to be equal to the Gorenstein dimension of that algebra. \newline I want to thank Steffen K\"onig for proofreading and useful suggestions.

\newpage

\section{Preliminaries}
In this article all algebras are finite-dimensional $K$-algebras, for an arbitrary field $K$, and all modules are finitely generated right modules, unless stated otherwise. We will also assume that our algebras will be connected. $J$ will always denote the Jacobson radical of an algebra. 
Recall that \emph{Nakayama algebras} are defined as algebras having the property that every indecomposable projective left or right module is uniserial, see for example \cite{SkoYam} chapter I.10. for more information on Nakayama algebras.
When talking about Nakayama algebras, we assume that they are given by quivers and relations (meaning that they are basic and split algebras). As explained in the introduction, this is not really a restriction since the dominant dimension is invariant under Morita equivalence and field extensions. A Nakayama algebra with an acyclic quiver is called \emph{LNakayama algebra} (L for line) and with a cyclic quiver \emph{CNakayama algebra} (C for circle).
The quiver of a CNakayama algebra:
$$Q=\begin{xymatrix}{ &  \circ^0 \ar[r] & \circ^1 \ar[dr] &   \\
\circ^{n-1} \ar[ur] &     &     & \circ^2 \ar[d] \\
\circ^{n-2} \ar[u] &  &  & \circ^3 \ar[dl] \\
   & \circ^5 \ar @{} [ul] |{\ddots} & \circ^4 \ar[l] &  }\end{xymatrix}$$
\newline
\newline
The quiver of an LNakayama algebra:
$$Q=\begin{xymatrix}{ \circ^0 \ar[r] & \circ^1 \ar[r] & \circ^2 \ar @{} [r] |{\cdots} & \circ^{n-2} \ar[r] & \circ^{n-1}}\end{xymatrix}$$

For connected CNakayama algebras with $n$ simple modules the simple modules are numbered from 0 to $n$-1 clockwise (corresponding to $e_i A$, the projective indecomposable modules at the point $i$).
$\mathbb{Z}/n$ denotes the cyclic group of order $n$ and $l_r(i)$ the length of the projective indecomposable right module at the point $i$ (so $l_r $ is a function from $\mathbb{Z}/n$ to the natural numbers). $l_l(i)$ gives the length of the projective indecomposable left module at $i$. For more facts about Nakayama algebras see for example the chapter about serial rings in \cite{AnFul} $\S$ 32.
Recall that the lengths of the projective indecomposable modules determine the Nakayama algebra uniquely. We often denote $l_r(i)$ by $c_i$ and $l_l(i)$ by $d_i$.
In the case of a non-selfinjective CNakayama algebra, one can order the $c_i$ such that $c_{n-1}=c_0+1$ and $c_i -1 \leq c_{i+1}$ for $0 \leq i \leq n-2$ and then $(c_0 , c_1 ,..., c_{n-1})$ is called the \emph{Kupisch series} of the Nakayama algebra.
A Nakayama algebra $A$ is selfinjective iff the $c_i=l_r(i)$ are all equal and the quiver of $A$ is a circle. Every indecomposable module of a Nakayama algebra is uniserial, which means that the chain of submodules of an indecomposable module coincides with its radical series. Thus one can write every indecomposable module of a Nakayama algebra as a quotient of an indecomposable projective module $P$ by a radical power of $P$. 
Two Nakayama algebras $A$ (with Kupisch series $(c_0,c_1,...,c_{n-1})$) and $B$ (with Kupisch series $(C_0,C_1,...,C_{m-1})$) are said to be in the same \textit{difference class}, if $n=m$ and $c_i \equiv_n C_i$ for all $i=0,1,...,n-1$. Given a Nakayama algebra with $n$ simple modules, the largest number of the $c_i$ minus the smallest number is less than $n$. Therefore there are only finitely many difference classes of Nakayama algebras with a fixed number of simple modules. $D:=Hom_K(-,K)$ denotes the $K$-duality of an algebra $A$ over the field $K$.
We denote by $S_i=e_iA/e_iJ$, $P_i=e_i A$ and $I_i=D(Ae_i)$  the simple, indecomposable projective and indecomposable injective module, respectively, at the point $i$. \newline
The \emph{dominant dimension} domdim($M$) of a module $M$ is defined as follows: Let \newline
\centerline{$0 \rightarrow M \rightarrow I_0 \rightarrow I_1 \rightarrow I_2 \rightarrow ... $} \newline
be a minimal injective coresolution of $M$. If $I_0$ is not projective, then we set domdim$(M):=0$ and otherwise \newline
\centerline{domdim($M$):=$\sup \{ n | I_i $ is projective for $i=0,1,...,n \}$+1.}
The \emph{codominant dimension} of a module $M$ is defined as the dominant dimension of the dual module $D(M)$. The dominant dimension of a finite dimensional algebra is defined as the dominant dimension of the regular module. 
So domdim($A$)$ \geq 1$ means that the injective hull of the regular module $A$ is projective. In case of domdim($A$) $ \geq 1$, there exists an idempotent $e$ such that $eA$ is a minimal faithful projective-injective module which is just the direct sum of all distinct indecomposable projective-injective modules.
Algebras with dominant dimension larger than or equal to 1 are called \emph{QF-3 algebras}.
All Nakayama algebras are QF-3 algebras (see \cite{Abr}, Proposition 4.2.2 and Propositon 4.3.3).
The Morita-Tachikawa correspondence says that an algebra $A$ has dominant dimension at least two iff $A$ is isomorphic to an algebra of the form $End_B(M)$, where $B$ is some algebra with generator-cogenerator $M$. In this case $B \cong eAe$, where $e$ is an idempotent such that $eA$ is a minimal faithful projective-injective right module. See for example \cite{Ta} for more details and proofs of the Morita-Tachikawa correspondence.
For more information on dominant dimensions and QF-3 algebras, we refer to \cite{Ta}.
The \emph{Gorenstein dimension} of an algebra $A$ is defined as the injective dimension of the right regular module and $A$ is called a \emph{Gorenstein algebra} in case the injective dimension of the right regular module is finite and coincides with the injective dimension of the left regular module. It is an open conjecture wheter the injective dimension of the left regular module always coincides with the injective dimension of the right regular module. This conjecture, called the Gorenstein symmetry conjecture, is a consequence of the famous finitistic dimension conjecture, which is true for representation-finite algebras, see for example \cite{ARS} in the conjectures section. Since we deal with Nakayama algebras (which are always representation-finite) in this article, we do not have to worry about the Gorenstein symmetry conjecture here and it is thus enough to calculate the right injective dimension of the regular module.
By an acyclic algebra we denote quiver algebras whose quiver is acyclic.

\section{Nakayama algebras}
In this article we prove, besides other things, that the dominant dimension of a non-selfinjective Nakayama algebra $A$ is bounded by $2s$, where $s$ is the number of nonisomorphic projective-injective indecomposable modules of $A$. Later we will provide examples which show that the number $2s$ is attained by some Nakayama algebras with $s$ nonisomorphic projective-injective indecomposable modules. Then we can correct and prove a sharpened version of a conjecture of Abrar, who conjectured that the dominant dimension of a Nakayama algebra with $n$ simple modules is bounded by $2n-3$ (see \cite{Abr}). We will in fact show that the correct bound is $2n-2$ and this value is attained in an example.

\subsection{Resolutions for Nakayama algebras} \label{minpro}
In this subsection results about Nakayama algebras will be collected.
$A$ will always be a Nakayama algebra  with $n$ simple modules and indices of primitive idempotents are integers modulo $n$.

Let $M:=e_iA/e_iJ^{k}$ be an indecomposable module of $A$. The projective cover of $M$ is obviously $e_i A$ and $\Omega(M)=e_i J^k$. Then $\Omega(e_i J^{k})=e_{i+k} J^{l_r(i)-k}$, since
\begin{align*}
top(e_i J^{k})&=e_i J^{k}/e_i J^{k+1} \cong S_{i+k},\\
\text{dim}(e_i J^{k})&=l_r(i)-k\ \text{and}\\
\text{dim}(e_{i+k} J^{l_r(i)-k})&=l_r(i+k)-(l_r(i)-k),
\end{align*}
which determines $\Omega(e_i J^{k})$ uniquely. To see this, recall that the submodules of $e_{i+k} A$ form a chain and dim($e_i J^{k}$)+dim $\Omega(e_i J^{k})$=dim($e_{i+k} A$)=$l_r(i+k)$. This motivates the following definition:
\begin{definition}
Define a function $f: \mathbb{Z}/n \rightarrow \mathbb{Z}/n$ as $f(x):=x+l_r(x)$.
\end{definition}
We note that the use of such a function is due to \cite{Gus}. The procedure to calculate syzygies succesively gives the following proposition, see also \cite{Gus}.
\begin{proposition}\label{minpro2}
The minimal projective resolution of $M=e_iA/e_iJ^{k}$ has the following form by repeating the above process ($f^e $ denotes the function $f$ taken $e$ times for a natural number $e \geq 0$): \newline
\begin{tiny}
\xymatrix{
\cdots \ar[r] & e_{f^r(j)} A  \ar[r] & e_{f^r(i)}A \ar[r] & \cdots \ar[r] & e_{f^2(j)} A                \ar@{->} `r/8pt[d] `/10pt[l] `^dl[llll]|{} `^r/3pt[dll] [dllll] \\
e_{f^2(i)} A\ar[r] & e_{f^1(j)} A \ar[r] & e_{f^1(i)} A \ar[r] & e_j A \ar[r] & e_i A \ar[r] & M \ar[r] & 0}
\end{tiny}
\end{proposition}
If we denote $e_x J^y$ for short by $(x,y) \in \mathbb{Z}/n \times \mathbb{N}$, then $\Omega(e_x J^y)=(x+y,c_x-y)$. Like this we can calculate the syzygies successively with a simple formula depending only on the Kupisch series of the Nakayama algebra.
Dually, we get a minimal injective coresolution (with $k=c_i$, if $M$ is projective):
We have soc($M$)=$S_{i+k-1}$ (the simple module corresponding to the point $i+k-1$). Therefore, the injective hull of $M$ is $D(Ae_{i+k-1})$ and we get $\Omega^{-1}(M)=D(J^k e_{i+k-1})$ and $\Omega^{-1}(D(J^k e_{i+k-1}))=D(J^{l_l(i+k-1)-k} e_{i-1})$, again by comparing dimensions and using that submodules form a chain.
\begin{definition}
Define $g:\mathbb{Z}/n \rightarrow \mathbb{Z}/n$ as $g(x):=x-l_l(x)$.
\end{definition}

\begin{proposition} \label{minpro3}
The minimal injective coresolution of $M$ looks like this by repeating the above process: \newline
\begin{center}
\begin{tiny}
\xymatrix@C=0.5cm@R=1cm{
0 \ar[r] & M  \ar[r] & D(Ae_{j-1}) \ar[r] & D(Ae_{i-1}) \ar[r] & D(Ae_{g(j-1)}) \ar[r] & D(Ae_{g(i-1)})  \ar@{->} `r/8pt[d] `/10pt[l] `^dl[lllll]|{} `^r/3pt[dll] [dlllll] \\
D(Ae_{g^2(j-1)}) \ar[r] & D(Ae_{g^2(i-1)}) \ar[r] & \cdots \ar[r] & D(Ae_{g^e(j-1)})  \ar[r] & D(Ae_{g^e(i-1)})  \ar[r] & \cdots}
\end{tiny}
\end{center}
\end{proposition}
If we denote $D(J^y e_x)$ for short by $[x,y] \in \mathbb{Z}/n \times \mathbb{N}$ then we get that $\Omega^{-1}(D(J^y e_x))=[x-y,d_x-y]$. Like this we can calculate the cosyzygies successively.

Now we specialize to selfinjective Nakayama algebras with Loewy length $k$.
We give the minimal projective resolution of a general indecomposable nonprojective module $M$ and a formula for $Ext^{i}(M,M)$ for arbitrary $i \geq 1$.
Without loss of generality, we can assume that $M=e_0 J^s$, for $1 \leq s \leq k-1$.
The minimal projective resolution of $M$ then looks like this: \newline
\begin{tiny}
\xymatrix@C=1.5cm@R=1cm{
\cdots \ar[r] & e_{(i+1)k+s}A \ar[r]^{L_{(i+1)k,s}} & e_{(i+1)k}A \ar[r]^{L_{ik+s,k-s}} & e_{ik+s} A \ar[r]^{L_{ik,s}} & e_{ik}A \ar@{->} `r/8pt[d] `/10pt[l] `^dl[llll]|{} `^r/3pt[dll] [dllll] \\
\cdots \ar[r] & e_k A \ar[r]^{L_{s,k-s}} & e_s A \ar[r]^{L_{0,s}} &  M \ar[r] & 0.}
\end{tiny}
\newline 
Here, we denote by $L_{x,y}$ the left multiplication by $w_{x,y}$, where $w_{x,y}$ is the unique path starting at $x$ and having length $y$.
If we apply the functor $Hom(-,M)$ to this minimal projective resolution (with $M$ deleted), then we get the complex: \newline
\begin{tiny}
\xymatrix@C=1.5cm@R=1cm{
0 \ar[r] & e_0 J^s e_s \ar[r]^{R_{s,k-s}} & e_0 J^s e_k \ar[r]^{R_{k,s}} & e_0 J^s e_{k+s} \ar[r]^{} & \cdots \ar@{->} `r/8pt[d] `/10pt[l] `^dl[llll]|{} `^r/3pt[dll] [dllll] \\
e_0 J^s e_{ik} \ar[r]^{R_{ik,s}} & e_0 J^s e_{ik+s} \ar[r]^{R_{ik+s,k-s}} & e_0 J^s e_{(i+1)k} \ar[r]^{R_{(i+1)k,s}} & e_0 J^s e_{(i+1)k+s}  \ar[r] & \cdots}
\end{tiny}
\newline
Here, $R_{x,y}$ is right multiplication by $w_{x,y}$.
We see that $R_{ik+s,k-s}=0$ for all $i \geq 1$, since we map paths of length at least $s$ to paths of length at least $k$ (and $J^k =0$).
Therefore, we have for all $i \geq 1$: \newline
$Ext^{2i-1}(M,M)=ker(R_{ik,s}) \neq 0$, iff there is a path of length larger than or equal to $k-s$ in $e_0 J^s e_{ik}$ and 
$Ext^{2i}(M,M)=e_0 J^s e_{ik+s} / Im(R_{ik,s})$.

\begin{Length of the indecomposable left modules}
\label{Length of the indecomposable left modules}
The length of the indecomposable projective left module $Ae_{x}$ at a vertex $x$ (and, therefore, the length of the indecomposable injective right module at $x$) satisfies:
$$l_l(x)=\inf \{k | k \geq l_r(x-k) \}.$$
The values of $l_l(x)$ are a permutation of the values of $l_r(x)$ and are determined uniquely by the lengths of the projective indecomposable right modules.
\end{Length of the indecomposable left modules}

\begin{proof}
See \cite{Ful} Theorem 2.2.
\end{proof}

\begin{Structure of indecomposable injective modules}
\label{Structure of indecomposable injective modules}
Let $M:= e_iA/e_iJ^m$ be an indecomposable module of the Nakayama algebra $A$ with $m=dim(M) \leq c_i$.
Then $M$ is injective iff $c_{i-1} \leq m$.
\end{Structure of indecomposable injective modules}

\begin{proof}
See \cite{AnFul} Theorem 32.6.
\end{proof}

The following theorem shows that the dominant dimension of a given Nakayama algebra depends only on its difference class:
\begin{theorem}
Let $A$ be a Nakayama algebra with $n$ simple modules and $M=e_iA/e_iJ^k$ be an $A$-module.
The dominant dimension of $M$ depends only on the difference class of $A$ and on the $i$ mod $n$ and $k$ mod $n$.
Especially, the dominant dimension of $A$ depends only on the difference class of $A$.
\end{theorem}
\begin{proof}
We may assume that $M$ is not injective.
First we see that in a given difference class of Nakayama algebras, $e_i A$ is injective iff $c_{i-1} \leq c_i$, so the position of the injective-projective modules doesn't depend on the choice of $A$ inside a given difference class.
In order to determine the dominant dimension of $M$, we calculate a minimal injective coresolution $(I_i)$ and the cosyzygies of $M$.
Note that $\Omega^{-1}(M)=D(J^k e_{i+k-1})$ and that calculating syzygies of modules of the form $[x,y]=D(J^ye_x)$ is done by $\Omega^{-1}[x,y]=[x-y,d_x-y]$. If $\Omega^{-j}(M)=D(J^pe_q)$, then $I_{j-1} \cong D(Ae_q)$. We see that all those calculations only depend on $i$, $k$ mod $n$ and the difference class (which determines the $d_i$ mod $n$) of the algebra.
Now there are two cases to consider: \newline 
\underline{Case 1}: $\Omega^{-j}(M) \neq 0$ for every $j \geq 1$. Note that the simple socles of the $I_i$ do not depend on the difference class and $i$ mod $n$ and $k$ mod $n$ as explained before. Thus the calculation of the dominant dimension of $M$ is also independent of the difference class and $i$ mod $n$ and $k$ mod $n$ . \newline
\underline{Case 2:} Assume now that $\Omega^{-j}(M)=0$ in one algebra of a given difference class, but $\Omega^{-j}(M)\neq 0$ in another algebra in the given difference class for a module $M$ of the form $e_iA/e_iJ^k$, for some $j\geq 1$. When $\Omega^{-j}(M)=0$ happens for some $j \geq 1$ for the first time, there must have been an $I_l$ with $l \leq j-1$, which is not projective. Otherwise we would have a minimal injective coresolution $(I_i)$, with the properties that all terms are also projective and that its ending has the following form:
$$\cdots \rightarrow I_{j-2} \stackrel{f}{\rightarrow} I_{j-1} \rightarrow 0.$$
Therefore, the surjective map $f$ between projective modules would be split, contradicting the minimality of the resolution.
So calculating the dominant dimension of $M$ involves only those terms $I_l$ for $1 \leq l \leq j-1$ in the minimal injective coresolution of $M$ until $\Omega^{-j}(M) = 0$ happens for the first time. But those terms in the injective coresolution depend only on the difference class of $A$ and $i$ mod $n$ and $k$ mod $n$ and so does the dominant dimension of $M$. 
\ \ \textcolor{white}{.}

\end{proof}

\begin{example}
We calculate the dominant dimension of a Nakayama algebra $A$ in the difference class of Nakayama algebras with Kupisch series $(c_0,c_1,c_2)=(3k+2,3k+2,3k+3)$, for $k \geq 0$.
First we calculate the dimension of the injective indecomposable modules:
\centerline{$l_l(0)= \inf \{ s \geq 3k+2 \mid s \geq l_r(-s) \} = 3k+2$ and likewise $l_l(1)=3k+3$ and $l_l(2)=3k+2$.}
Thus $(d_0,d_1,d_2)=(3k+2,3k+3,3k+2)$.
With soc($e_1A$)=$e_1J^{3k+1} \cong S_2$, it follows that $e_1 A$ embeds into $D(Ae_2)$. But, since $e_1 A$ and $D(Ae_2)$ have the same dimension, they are isomorphic. \newline
With soc($e_2 A$) = $e_2J^{3k+2} \cong S_1 $, it follows that $e_2 A$ embeds into $D(Ae_1)$ and as above both are isomorphic, because they have the same dimension.
Thus the projective-injective indecomposable modules are $e_1A \cong D(Ae_2)$ and $e_2A \cong D(Ae_1)$.
Now it is enough to look at an injective coresolution of $e_0A$.
Since soc($e_0 A$) = $e_0 J^{3k+1} \cong S_1$, $e_0A$ embeds into $D(A e_1)$ with cokernel equal to $D(J^{3k+2} e_1)=[1,3k+2]$. Then $\Omega^{-1}([1,3k+2])=[1-(3k+2),d_1-(3k+2)]=[2,1]$ and $\Omega^{-1}([2,1])=[2-1,d_2-1]=[1,3k+1]$ and $\Omega^{-1}([1,3k+1])=(1-(3k+1),d_1-(3k+1))=[0,2]$.
The minimal injective coresolution of $e_0 A$ starts as follows:
$$0 \rightarrow e_0 A \rightarrow D(Ae_1) \rightarrow D(Ae_2) \rightarrow D(Ae_1) \rightarrow D(Ae_0) \rightarrow \cdots .$$
Since $D(Ae_0)$ is not projective, the dominant dimension of $e_0 A$ is equal to 3, as is the dominant dimension of $A$, since $e_0 A$ is the only indecomposable projective and not injective module.
Note that if $A$ has Kupisch series $(2,2,3)$, then $D(J^2 e_0)=[0,2]$=0, while for $k \geq 1$, that module is nonzero.
Also note that the Gorenstein dimension is not independent of the difference class of the Nakayama algebra:
If $A$ has Kupisch series $(2,2,3)$, then, by the above, the Gorenstein dimension is equal to the dominant dimension and finite.
But, if $A$ has Kupisch series $(3k+2,3k+2,3k+3)$ for a $k \geq 1$, then continuing as above, one gets: $\Omega^{-1}([0,2])=[1,3k],\ \Omega^{-1}([1,3k])=[1,3],\ \Omega^{-1}([1,3])=[1,3k+2]=\Omega^{-1}(e_0A)$, and the resolution gets periodic and is, therefore, infinite.
\end{example}

\subsection{Gorenstein-projective modules}
In this section, $A$ denotes a finite dimensional algebra.
See \cite{Che} Section 2, for an elementary introduction to Gorenstein homological algebra.
We take our definitions and lemmas from this source.
\begin{definition}
A complex $P^{\bullet} : ... \rightarrow P^{n-1} \stackrel{d^{n-1}}{\rightarrow} P^n \stackrel{d^{n}}{\rightarrow} P^{n+1} \rightarrow ...$ of projective $A$-modules is called totally acyclic, if it is exact and the complex $Hom( P^{\bullet} ,A)$ is also exact.
An $A$-module $M$ is called Gorenstein-projective, if there is a totally acyclic complex of projective modules such that $M=ker(d^0)$. We denote by $A$-gproj the full subcategory of mod-$A$ of Gorenstein-projective modules and we denote by $^{\perp}A$ the full subcategory of mod-$A$ of all modules $N$ with $Ext^{i}(N,A)=0$, for all $i \geq 1$. ${D(A)}^{\perp}$ denotes the full subcategory of mod-$A$ of all modules $N$ with $Ext^{i}(D(A),N)=0$ for all $i \geq 1$.
\end{definition}

\begin{lemma}
(see \cite{Che} Corollary 2.1.9. and 2.2.17.) \newline
Let $A$ be a finite dimensional algebra and $M$ an $A$-module. 
\begin{enumerate}
\item $A$-gproj $\subseteq \, ^{\perp}A$.  
\item An $A$-module $N$ is in $A$-gproj, in case there is an $n$, such that $Ext^{i}(N,A) =0$, for all $i=1,...,n$, and $\Omega^{n}(N)=N$. 
\item If $Ext^{i}(N,A)=0$ for all $i=1,...,d$ and $\Omega^{d}(N)$ is Gorenstein-projective, then also $N$ is Gorenstein-projective.
\end{enumerate}
\end{lemma}

\begin{lemma}
\label{gorgleich}
If $A$ is a Nakayama algebra, then $A$-gproj $=\, ^{\perp}A$. 
\end{lemma}

\begin{proof}
We know that $A$-gproj $\subseteq\, ^{\perp}A$.
Now let $M \in\, ^{\perp}A$ with $M$ indecomposable. Since all syzygies over a Nakayama algebra of an indecomposable module are also indecomposable and since there is only a finite number of indecomposable modules, there exist numbers $k,n$ with : $\Omega^{n}(\Omega^{k}(M))=\Omega^{k}(M)$.
Since we also have $\Omega^{k}(M) \in\, ^{\perp}A$ by the formula $Ext^{i}(\Omega^{k}(M),A)=Ext^{i+k}(M,A)=0$, we know that $\Omega^{k}(M)$ is Gorenstein-projective by (2) of the above lemma.
Now by (3) of the above lemma also $M$ is Gorenstein-projective.
\ \ \textcolor{white}{.}
\end{proof}

\begin{lemma}
\label{Gorenstein}
An indecomposable injective and Gorenstein-projective module is projective.
\end{lemma}
\begin{proof}
By definition, a Gorenstein-projective module $M$ embeds in a projective module. If this module $M$ is additionally injective, this embedding splits and $M$ is itself projective.
\end{proof}

\subsection{CoGen-dimension and dominant dimension}

\begin{definition}
For a finite dimensional algebra $A$ and a module $M$ we define $\phi_M$ as
$$\phi_M:= \inf \{ r \geq 1 | Ext_{A}^{r}(M,M)\neq 0 \}, $$
with the convention $\inf(\emptyset)= \infty$.
We call a module $M$ which is a generator and a cogenerator for short a \emph{CoGen}.
We also define $\Delta_A:= \sup \{ \phi_M | M$ is a nonprojective CoGen $\}.$
\end{definition}

We remark that for a non-selfinjective algebra $A$ \newline $\Delta_A=  \inf \{ r \geq 1 | Ext_{A}^{r}(D(A),A) \neq 0 \}$, and for a selfinjective algebra $A$ \newline $\Delta_A=  \sup \{ \phi_M | M$ is a non-projective indecomposable A-module$ \}$.

\begin{Theorem of Mueller}
(see \cite{Mue})
If $M$ is a CoGen of $A$, then
the dominant dimension of $B:=End_A (M)$ is equal to $\phi_M+1$.
\label{Mueller}
\end{Theorem of Mueller}
 
\begin{Nakayama conjecture}
The Nakayama conjecture states that every non-selfinjective finite dimensional algebra has finite dominant dimension.
\label{2conj}
\end{Nakayama conjecture}
As a corollary of Mueller's theorem, the Nakayama conjecture is equivalent to the finiteness of $\Delta_A$, for every finite dimensional algebra $A$.

\begin{Yamagata conjecture}
\label{3conj}
Yamagata (in \cite{Yam}) states the even stronger conjecture that the dominant dimensions of non-selfinjective algebras with a fixed number of simple modules are bounded by a function of the number of simple modules of $A$. More precisly this means that $domdim(A) \leq f(n)$ for any non-selfinjective algebra with $n$ simple modules and a finite function $f$.
\end{Yamagata conjecture}
In this section, we will prove Yamagata's conjecture in case $eAe$ is a Nakayama algebra or a quiver algebra with an acyclic quiver, when $eA$ is the minimal faithful injective-projective $A$-module.

\begin{lemma}
Let $A$ be a non-selfinjective connected algebra of finite Gorenstein dimension $g=injdim(A)$.
Then $\Delta_A \leq g$.
\end{lemma}
\begin{proof}
We have
\begin{align*}
g=injdim(A)&=projdim(D(A))=\\
\sup \{ r \geq 1 | Ext_{A}^{r}(D(A),A) \neq 0 \} &\geq \inf \{ r \geq 1 | Ext_{A}^{r}(D(A),A) \neq 0 \}\\
&= \Delta_A,
\end{align*}
where we used $projdim(M)=sup \{ r \geq 1 | Ext_{A}^{r}(M,A) \neq 0 \}$, in case $M$ has finite projective dimension.
\ \ \textcolor{white}{.} 
\end{proof}

The following generalizes and gives an easier proof of Theorem 1.2.3 of \cite{Abr}, which states (3) of the following Corollary.
\begin{corollary}
\begin{enumerate}
\item Let $A$ be an connected acyclic algebra with $d \geq 2$ simple modules.  Then $gldim(A) \leq d-1$ and, therefore, $\Delta_A \leq d-1$. 
\item Let $A$ be a QF-3 algebra with $s$  projective-injective indecomposable modules such that $eAe$ is acyclic, where $eA$ is the minimal faithful injective-projective module. Then domdim($A$) $\leq s$. 
\item Let $A$ be an acyclic algebra with $s$ indecomposable injective-projective modules, then domdim($A$) $\leq s$. 
\item For an LNakayama algebra $A$ with $n$ simple modules, the following holds:  $\Delta_A \leq n-1$.
\end{enumerate}
\end{corollary}

\begin{proof}
\begin{enumerate}
\item For an elementary proof of $gldim(A) \leq d-1$, see e.g. \cite{Farn}. Then $\Delta_A \leq d-1$ follows from the previous lemma, since the equality $gldim(A)=injdim(A)$ holds, in case $A$ has finite global dimension. 
\item In case $A$ has dominant dimension equal to one, there is nothing to show. Thus assume that $A$ has dominant dimension at least two. Then $A \cong End_B(M)$ for some algebra $B$ and a generator-cogenerator $M$, by the Morita-Tachikawa correspondence. But $B \cong eAe$ is acyclic with $s$ simple modules and the result then follows from Mueller's theorem since $domdim(A) \leq \Delta_B +1 \leq s$ by (1).
\item This holds, since with $A$ also $eAe$ is acyclic for every idempotent $e \in A$. 
\item This is clear by (3), since LNakayama algebras are acyclic. 
\end{enumerate}
\end{proof}

We give the following example of  a nonacyclic algebra $\Lambda$ such that $C=e \Lambda e$ is acyclic to show that the above is really a generalisation of Theorem 1.2.3 of \cite{Abr}.
\begin{example}
Take any acyclic endowild algebra $C$ (this means that for every finite dimensional algebra $R$, there is a finite dimensional $C$-module $N$ with $R \cong End_C(N)$. Examples of such algebras $C$ are wild hereditary algebras over an algebraically closed field, see \cite{SimSko3} page 329) and a basic $C$-module $M$ such that $End_C(M)$ is not acyclic.
We claim that then the algebra $\Lambda=End_C(B(C \oplus D(C) \oplus M))$ is not acyclic, where $B(X)$ denotes the basic version of a module $X$. Denoting here by $f$ the projection from $B(C \oplus D(C) \oplus M)$ onto $B(M)$, the algebra $f \Lambda f \cong End_C(B(M))$ is not acyclic and thus $\Lambda$ is not acyclic. On the other hand, denoting by $e \in \Lambda$ idempotent such that $e \Lambda$ is the minimal faithful projective-injective right module, the algebra $e \Lambda e \cong C$ is acyclic.
\end{example}

For the main lemma, recall the following result:
\begin{lemma}
Let $A$ be a finite dimensional algebra, $N$ be an indecomposable $A$-module and $S$ a simple $A$-module.
Let $(P_i)$ be a minimal projective resolution of $N$ and $(I_i)$ a minimal injective coresolution of $N$. 
\begin{enumerate}
\item For $l \geq 0$, $Ext^{l}(N,S) \neq 0$ iff $S$ is a quotient of $P_l$. 
\item For $l \geq 0$, $Ext^{l}(S,N) \neq 0$ iff $S$ is a submodule of $I_l$.
\end{enumerate}
\end{lemma}
\begin{proof}
See \cite{Ben} Corollary 2.5.4. 
\end{proof}

\begin{Mainlemma}
Let $A$ be a finite dimensional non-selfinjective Nakayama algebra with $n \geq 2$ simple modules. Let $N$ be an $A$-module and $S$ a simple $A$-module.
\begin{enumerate}
\item Assume that $Ext^{l}(N,S) \neq 0$ for some $l \geq 1$. Then $\inf \{ s \geq 1 | Ext^{s}(N,S) \neq 0 \} \leq 2n-2$.
\item Assume that $Ext^{l}(S,N) \neq 0$ for some $l \geq 1$. Then $\inf \{ s \geq 1 | Ext^{s}(S,N) \neq 0 \} \leq 2n-2$.
\end{enumerate}
\end{Mainlemma}
\begin{proof}
We only prove (1) since (2) follows dually. \newline
We can assume that $Ext^{1}(N,S)=0$, since the result is obvious in case $Ext^{1}(N,S)\neq0$. \newline
So in the following we look at the problem of determining the smallest possible finite value $s \geq 2$ with respect to the following properties: \newline
$Ext^{s}(N,S) \neq 0$, but $Ext^{i}(N,S) = 0$, for $i=1,...,s-1$, for an indecomposable module $N$. $Ext^{s}(N,S) \neq 0$ simply means that in the minimal projective resolution $(P_i)$ of $N$ there is a direct summand of $P_s$ isomorphic to the projective cover of $S=S_r:=e_rA/e_rJ$ by the previous lemma.
By \hyperref[minpro2]{\ref*{minpro2}} the minimal projective resolution has the form:
\begin{tiny}
\xymatrix{
\cdots \ar[r] & e_{f^r(j)} A  \ar[r] & e_{f^r(i)}A \ar[r] & \cdots \ar[r] & e_{f^2(j)} A                \ar@{->} `r/8pt[d] `/10pt[l] `^dl[llll]|{} `^r/3pt[dll] [dllll] \\
e_{f^2(i)} A\ar[r] & e_{f^1(j)} A \ar[r] & e_{f^1(i)} A \ar[r] & e_j A \ar[r] & e_i A \ar[r] & N \ar[r] & 0}
\end{tiny} \newline \newline
$\textit{Claim 1.}$ $f$ is not surjective. \newline
$\textit{Proof:}$ If $f$ were surjective, it would be bijective and, because of $soc(e_i A)= S_{f(i)-1}$, for every $i$, $A$ would be selfinjective (see \cite{SkoYam} Chapter IV. Theorem 6.1.), contradicting our assumption that $A$ is not selfinjective. So Claim 1 is proved. \newline \newline
Now, $Ext^u(N,S) \neq 0$, for some $u \geq 1$, tells us $e_rA \cong P_u$. \newline \newline
$\textit{Claim 2.}$ The smallest index $i$ with $e_rA \cong P_i$ must be smaller than or equal to $2n-2$. \newline
$\textit{Proof:}$ Since $f$ is a mapping from a finite set to a finite set, there is a minimal number $w$ with $Im(f^w)=Im(f^{w+1})$. Define $X:=Im(f^w)$. Note that the cardinality of $X$ is smaller than or equal to $n-w$, since $f : \mathbb{Z}/n \rightarrow \mathbb{Z}/n$ is not surjective and therefore the number of elements in $Im(f^{i})$ decreases by at least 1 as long as $i < w$ in the sequence $Im(f) \supset Im(f^2) \supset \cdots \supset Im(f^w)$.\newline
$f$ is a bijection from $X$ to $X$, since $f\big|_X : X \rightarrow X$ is surjective (and $X$ a finite set).
When we have reached
$$\cdots \rightarrow e_{f^{w+1}(j)} A \rightarrow e_{f^{w+1}(i)} A \rightarrow e_{f^w(j)} A \rightarrow e_{f^w(i)} A \rightarrow \cdots,$$
$f$ acts as a cyclic permutation on the $\{f^l(i)\}$ and $\{f^l(j)\}$ for $l \geq w$. Note that after the term $e_{f^{w+n-w-1}(i)} A = e_{f^{n-1}(i)} A=P_{2n-2}$ (recall that $X$ has cardinality at most $n-w$) some index $q \in X$ must exist such that $P_{2n-1}=e_qA$ and this indecomposable projective module $e_qA$ is also isomorphic to $P_{u}$ for an $u \leq 2n-2$.
Therefore, the index $r$ must occur in the first $2n-2$ terms, since later there are only indices which already occurred before.
\end{proof}

\begin{theorem}
\label{best result}
Let $A$ be an algebra with dominant dimension larger than $1$ and the property that $eA$ is a minimal faithful projective injective module and $eAe$ is a Nakayama-algebra with $s$ simple modules. Then: \newline
$$domdim(A)\leq \Delta_{eAe} +1 \leq 2s.$$
Especially, Yamagata's conjecture is true in this special case.
\end{theorem}
\begin{proof}
We first prove the following lemma:
\begin{lemma}
\label{Grenze1}
For a Nakayama algebra with $n$ simple modules, the following holds: $\Delta_A \leq 2n-1$.
\end{lemma}
We split the proof of this lemma in two cases: one case of a non-selfinjective Nakyama algebra and in the other case the Nakayama algebra is selfinjective.

\begin{Case 1}
For a non-selfinjective CNakayama algebra $A$ with $n$ simple modules, $\Delta_A \leq 2n-1$.
\end{Case 1}

\begin{proof}
There is an injective indecomposable module $I$ of the form $I=eA/soc(eA)$: The  module $I=eA/soc(eA)$ is injective, when $i$ is chosen such that $c_{i-1} \leq c_i -1$ and $e:=e_i$, which is always possible when $A$ is not selfinjective. This module is not Gorenstein-projective by \hyperref[Gorenstein]{Corollary \ref*{Gorenstein}}, since it is injective, but not projective.
By \hyperref[gorgleich]{Lemma \ref*{gorgleich}} $I$ is not in $^{\perp}A$.
Therefore, there is a smallest index $k \geq 1$ with $Ext^{k}(I,A) \neq 0$. Thus also $Ext^{k}(D(A),A) \neq 0$ which implies the lemma, if we show that $k \leq 2n-1$.
If $Ext^{1}(I,A) \neq 0$, there is nothing to prove. So we assume that $Ext^{1}(I,A) =
 0$. Denote soc(eA) by $S=S_r$ (which means that the projective cover of S is $e_r A$). In general we have $Ext^{i}(S,A)=Ext^{i}(\Omega(I) ,A) =  Ext^{i+1}(I ,A)$. Therefore, we will look in the following for the smallest index s with $Ext^{s}(S,A) \neq 0$. But by the main lemma $\inf \{ s \geq 1 | Ext^{s}(S,A) \neq 0 \} \leq 2n-2$. \newline
Because of $Ext^{i}(S,A)=Ext^{i}(\Omega(I) ,A) = Ext^{i+1}(I ,A)$ and $2n-2+1=2n-1$ we proved case 1.

\end{proof}

For the next case, recall the results about calculating minimal projective resolutions and $Ext^{i}(M,M)$ for an indecomposable module $M$ in a selfinjective Nakayama algebra from section 2.1.
\begin{Case 2}
A selfinjective Nakayama algebra $A$ with $n$ simple modules, satisfies: 
$$\Delta_A \leq 2n-1.$$
\label{selftheo}
\end{Case 2}
\begin{proof}
To prove this, we have to show that $\phi_M \leq 2n-1$ for all nonprojective indecomposable modules $M$. \newline
We can assume that $A$ has Loewy length $k$ and $M=e_0J^{s}$, with $1\leq s \leq k-1$.
We consider two cases: \newline
\underline{First case:} $k$ is a zero divisor in $\mathbb{Z}/n$.
Then there is a $q$ with $kq\equiv_n n \equiv_n 0$ and $1 \leq q \leq n-1$.
We know that $\Omega^{2i}(M)=e_{ik}J^s$ and, therefore, $\Omega^{2q}(M)=e_{qk}J^s = e_0 J^s =M$. Consequently, $Ext^{2q}(M,M)=\underline{Hom}(\Omega^{2q}(M),M) = \underline{Hom}(M,M) \neq 0.$ \newline
\underline{Second case:} $k$ is not a zero divisor in $\mathbb{Z}/n$ and, therefore, a unit.
We have $Ext^{2i-1}(M,M)\newline =ker(R_{ik,s}) \neq 0$, iff there is a path of length larger than or equal to $k-s$ in $e_0 J^s e_{ik}$. But, for $i=1,...,n$, the integers $ik$ are all different from one another mod $n$. This is why there surely is a path of length larger than or equal to $k-s$ in $e_0 J^s e_{ik}$ for some $i \leq n$. 
\end{proof}

Now we return to the proof of \hyperref[best result]{ \ref*{best result}}: \newline
Combining Case 1 and Case 2, we have proved the lemma\hyperref[Grenze1]{ \ref*{Grenze1}}. To get a proof of theorem \hyperref[best result]{ \ref*{best result}}, we use Mueller's theorem and the fact that the number of nonisomorphic indecomposable projective-injective modules equals the number of simple modules of $eAe$ to get that 
$$domdim(A)\leq \Delta_{eAe} +1 \leq 2s.$$ This finishes the proof of  \hyperref[best result]{ \ref*{best result}}. 

\end{proof}

\subsection{Dominant dimension of monomial algebras}
Recall that a \emph{monomial algebra} is a quiver algebra $KQ/I$ with admissible ideal $I$ generated by zero relations.
We will prove the bound of result 2 in this section:
\begin{theorem}
\label{Grenze2new}
Let $B$ be a non-selfinjective monomial algebra with $s$ projective-injective indecomposable modules. Then the dominant dimension of $B$ is bounded above by $2s$.
\end{theorem}
Note that $s$ in the previous theorem is always bounded by $n-1$, when $n$ denotes the number of simple $B$-modules.
We will also show in the next section that there is a Nakayama algebra such that the maximal value $2(n-1)$ (if $B$ has $n$ simple modules) is attained. Therefore the maximal possible value of the dominant dimension of a monomial algebra with $n$ simple modules is $2(n-1)$.

\begin{proposition}\label{eBe is Nak}
Let $A$ be a non-selfinjective monomial algebra with minimal faithful projective-injective module $eA$, then $eAe$ is a Nakayama algebra.

\end{proposition}
\begin{proof}
Let $e=e_1+e_2+...+e_r$ be a decomposition of the idempotent $e$ into primitive orthogonal idempotents $e_i$.
To show that $eAe$ is a Nakayama algebra, it is enough to show that its quiver is a directed line or a directed circle, which equivalently can be formulated as $dim(rad(e_iAe)/rad^{2}(e_iAe))=0$ or $=1$ for every $i$ and dually $dim(rad(eAe_i)/rad^{2}(eAe_i))=0$ or $=1$. We show $dim(rad(e_iAe)/rad^{2}(e_iAe))=0$ or $=1$ in the following, while the dual property can be proven dually or by going over to the opposite algebra (which, of course, is still monomial). Since $eA$ is injective, all modules $e_i A$ are injective and thus have a simple socle.
Note that $rad(eAe)=eJe$ and thus $rad(e_iAe)=e_iJe$ and $rad^{2}(e_iAe)=e_iJeJe$.
Assume that the dimension of $e_iJe/e_iJeJe$ is at least two for some $i$. Since $A$ is monomial and $e_iA$ has a simple socle, there can be only one arrow $\alpha$ starting at $e_i$. The target of $\alpha$ can not be an idempotent of the form $e_s$ for $1 \leq s \leq r$, or else  $e_iJe/e_iJeJe$ would be at most one-dimensional.
Thus there exists two paths $r_1=\alpha p_1$ and $r_2= \alpha p_2$ of smallest length starting at $e_i$ and going to $e_x$ and $e_y$ respectively, where $e_xA$ and $e_yA$ are summands of $eA$.
But this contradicts the fact that the socle of $e_iA$ is simple and $A$ being monomial, since there is no commutativity relation so that there are two different socle elements having the paths $r_1$ and $r_2$ as factors. Thus $dim(rad(e_iAe)/rad^{2}(e_iAe))=0$ or $=1$ and $eAe$ has to be a Nakayama algebra.
\end{proof}

We come now to the proof of \ref{Grenze2new}:
\begin{proof}
Assume now that $B$ is a monomial algebra and additionally assume that $B$ has dominant dimension at least one, since there is nothing to show otherwise. Thus there is an idempotent $e$ such that $eB$ is minimal faithful projective-injective.
Therefore, \ref{Grenze2new} follows from \hyperref[best result]{\ref*{best result}} and the theorem of Mueller, since $A:=eBe$ is a Nakayama algebra with $s$ simple modules and $\Delta_A \leq 2s-1$. So we have by Mueller's theorem: domdim($B$) = $\Delta_A +1 \leq 2s$. 
\end{proof}

\begin{corollary} \label{Grenze2}
Let $B$ be a non-selfinjective Nakayama algebra with $n$ simple modules. Then the dominant dimension of $B$ is bounded by $2n-2$.
\end{corollary}
\begin{proof}
This follows from \ref{Grenze2new}, since any Nakayama algebra is a monomial algebra.
\end{proof}
This corrects and proves a conjecture of Abrar, who conjectured that the maximal value is $2n-3$ (see \cite{Abr} Conjecture 4.3.21). In section 3, we show that $2n-2$ is the optimal bound.

\subsection{Finitistic dominant dimension of Nakayama algebras}
Using the main lemma, we show in this section that we can give a bound of the finitistic dominant dimension for Nakayama algebras.
\begin{definition}
The \textit{finitistic dominant dimension} of a finite dimensional algebra $A$ is 
$$\text{fdomdim}(A):=\text{sup} \{ \text{domdim}(M) | \text{domdim}(M) < \infty \}$$
\end{definition}
\begin{example}
If $A$ has global dimension $g$, then fdomdim($A) \leq g$, since for every noninjective module $M$  domdim($M) \leq $injdim($M) \leq g$ holds.
\end{example}
The following theorem gives again the bound $2n-2$ for the dominant dimension of Nakayama algebras.
\begin{theorem}
\label{findom}
Let $A$ be a non-selfinjective Nakayama algebra with $n \geq 2$ simple modules. Then fdomdim($A) \leq 2n-2$.
\end{theorem}
\begin{proof}
Clearly we can assume that $A$ is a CNakayama algebra, since an LNakayama algebra has global dimension at most $n-1$. So assume now that $A$ is a CNakayama algebra and $M$ an indecomposable $A$-module with finite dominant dimension. Note that domdim($M)= \inf \{ i | Ext^{i}(S,M) \neq 0$ for a simple module $S$ with nonprojective injective hull$\}$.
We can assume that $M$ has dominant dimension larger than or equal to 1.
Let $S$ be a simple module with nonprojective injective envelope such that $Ext^{i}(S,M) \neq 0$ for an $i \geq 1$.
Then by the main lemma $\inf \{ s \geq 1 | Ext^{s}(S,N) \neq 0 \} \leq 2n-2$ and thus domdim($M) \leq 2n-2$.
\end{proof}
\begin{example} \label{examplefindomdim}
Take the CNakayama algebra $A$ with Kupisch series $(3s+1,3s+2,3s+2), s \geq 1$. We first calculate the Gorenstein dimension and the dominant dimension of $A$ and then the finitistic dominant dimension of $A$.
First note that $e_1 A \cong D(A e_2)$ is injective. Also $e_2 A \cong D(A e_0)$ is injective. 
The only noninjective indecomposable projective module is then $e_0 A$ and the only nonprojective injective indecomposable module is $D(Ae_1)$. We have the following injective coresolution: 
$$0 \rightarrow e_0 A \rightarrow D(A e_0) \rightarrow  D(A e_2) \rightarrow D(A e_1) \rightarrow 0.$$
Thus the dominant dimension and the Gorenstein dimension of $A$ are both equal to 2.
Now take an indecomposable module $M=e_aA/e_aJ^k$ and calculate the minimal injective presentation of $M$ using the method from \hyperref[minpro]{ \ref*{minpro}}: $0 \rightarrow M \rightarrow D(A e_{a+k-1}) \rightarrow D(A e_{a-1})$.
Thus $M$ has dominant dimension larger than or equal to 2 iff $a+k-1 \in \{0,2 \}$ mod $3$ and $a-1 \in \{0,2 \}$ mod $3$ iff ($a=0$ mod $3$ and $k \in \{0,1 \}$ mod $3$) or ($a=1$ mod $3$ and $ k \in \{0,2 \}$ mod $3$).
The following table gives the relevant values of the dominant dimensions: \newline
\begin{center}
\begin{tabular}{l | {c}r}
a             & 0 & 1 \\
\hline
$k \equiv_3 0$ & 4 & 2  \\
$k \equiv_3 1$            & 2 & -  \\
$k \equiv_3 2$           & - & 3  \\
\end{tabular}
\end{center}
Thus the finitistic dominant dimension equals 4, while the finitistic dimension equals the Gorenstein dimension which is 2. This example also shows that in general the finitistic dominant dimension will be larger than the dominant dimension of a Nakayama algebra.
\end{example}

\section{Nakayama algebras which are Morita algebras and their dominant and Gorenstein dimension}
In this section we calculate the dominant dimension of all Nakayama algebras that are Morita algebras and give the promised example of a non-selfinjective Nakayama algebra having $n$ simple modules and dominant dimension $2n-2$. We also show how to calculate the Gorenstein dimension of such algebras and give a surprising interpretation of the dominant and Gorenstein dimension for gendo-symmetric Nakayama algebras.
\subsection{Calculating the dominant dimensions of Nakayama algebras that are Morita algebras}
\begin{definition}
A finite dimensional algebra $B$ is called a Morita algebra, if it is isomorphic to the endomorphism ring of a module $M$, which is a generator of a selfinjective algebra $A$ (see \cite{KerYam}).
If $A$ is even symmetric, then $B$ is called a gendo-symmetric algebra (see \cite{FanKoe} and \cite{Mar2} for other characterisations).
\end{definition}

The following is a special case of a result of Yamagata in \cite{Yam2}. 
\begin{Theorem(Yamagata)}
Let $A$ be a nonsemisimple selfinjective Nakayama algebra with Loewy length $w$ and  $M=\bigoplus\limits_{i=0}^{n-1}{e_i A} \oplus \bigoplus\limits_{i=1}^{r}{e_{x_i} A / e_{x_i} J^{k_i}}$. Then $B:=End_A (M)$ is a basic non-selfinjective Nakayama algebra, iff all the $x_i$ are pairwise different and $k_i = w-1$ for all $i \in \{ 1,...,r \}$ and $r \geq 1$.
\end{Theorem(Yamagata)}

Keep this notation for $B$ and call points of the form $x_i$ special points.

\begin{proposition}
Let $r \geq 1$.
Let $A$ be a selfinjective Nakayama algebra with Loewy length $w$ and $n$ simple modules. Let
$M=\bigoplus\limits_{i=0}^{n-1}{e_i A} \oplus \bigoplus\limits_{i=1}^{r}{e_{x_i} A / e_{x_i} J^{w-1}}$ with the $x_i$ different for all $i \in \{ 1,...,r \}$. Then:
\begin{align*}
domdim(B)&=\phi_M +1\\
&=\inf \{ k \geq 1 | \exists x_i , x_j : Ext^{k}(e_{x_i} A / e_{x_i} J^{w-1}, e_{x_j} A / e_{x_j} J^{w-1}) \neq 0 \}+1\\
&=\inf \{ k \geq 1| \exists x_i , x_j : x_j +w-1 \equiv_n x_i +[\frac{k+1}{2}]w-g_k  \}+1.
\end{align*}
Here, we set $g_k=1$, if $k$ is even, and $g_k=0$, if $k$ is odd. $[l]$ is equal to $l$, if $l$ is an integer, and otherwise equal to the smallest integer larger than $l$ (for example $[1.5]=2$).
\end{proposition}

\begin{proof}
Note that the first equality is by Mueller's theorem and we just have to show the last equality.
Lemma 1.3.7 says that for a module $M$ and a simple module $S$ $Ext^{i}(M,S) \neq 0$ iff S is a direct summand of the top of the module $P_i$, where $P_i$ is the $i$-th term in a minimal projective resolution of $M$.
Note that since $\Omega^{1}$ is a stable equivalence:
$$Ext^{k}(e_{x_i} A / e_{x_i} J^{w-1}, e_{x_j} A / e_{x_j} J^{w-1})= Ext^{k}(e_{x_i} J^{w-1} ,e_{x_j} J^{w-1}),$$
which is what we want to calculate. \newline
Observe that $e_{x_i} J^{w-1} \cong S_{x_i +w-1}$ is a simple module.
The minimal projective resolution of $e_{x_i} J^{w-1}$ then looks like this: \newline
\begin{tiny}
\xymatrix{
\cdots \ar[r] & e_{x_i+(l+1)w+w-1}A \ar[r] & e_{x_i+(l+1)w}A \ar[r] & e_{x_i+lw+w-1} A \ar[r] &  e_{x_i+lw}A                \ar@{->} `r/8pt[d] `/10pt[l] `^dl[llll]|{} `^r/3pt[dll] [dllll] \\
\cdots \ar[r] & e_{x_i+w} A \ar[r] & e_{x_i+w-1} A \ar[r] & e_{x_i} J^{w-1} \ar[r] & 0 }
\end{tiny} 
\newline \newline
Thus the $k$-th term in the minimal projective resolution of $e_{x_i}J^{w-1}$ is equal to
$$P_k= e_{x_i +[\frac{k+1}{2}]w-g_k }A.$$
Then $Ext^{k}(e_{x_i} A / e_{x_i} J^{w-1}, e_{x_j} A / e_{x_j} J^{w-1}) \neq 0$, iff $x_j +w-1 \equiv_n x_i +[\frac{k+1}{2}]w-g_k$, for a $k \geq 1$. 
\end{proof}

\begin{corollary}
\label{formgendo}
Let $A$ be a symmetric Nakayama algebra with Loewy length $w \equiv_n 1$ and $n$ simple modules. Let
$M=\bigoplus\limits_{i=0}^{n-1}{e_i A} \oplus \bigoplus\limits_{i=1}^{r}{e_{x_i} A / e_{x_i} J^{w-1}}$ with the $x_i$ different for all $i \in \{ 1,...,r \}$. Then for $B=End_A(M):$
\begin{align*}
\text{domdim}(B)&= 2\inf \{ s \geq 1 \mid \exists i,j: x_i + s  \equiv_n x_j  \}
\end{align*}
So the dominant dimension is just twice the (directed) graph theoretical minimal distance of two special points which appear in $M$.
\end{corollary}
\begin{proof}

The formula takes for $w \equiv_n 1$ an especially nice form:
$$\text{domdim}(B) =  \inf \{ k \geq 1 \mid \exists i,j: x_i + [\frac{k+1}{2}]-g_k \equiv_n x_j  \}+1.$$
For $k=2s+1$ and $k=2s+2$ the value of $[\frac{k+1}{2}]-g_k$ is the same.
This means that the infimum is attained at an odd number of the form $k=2s-1$ and the formula simplifies to
\begin{align*}
\text{domdim}(B) &=  \inf \{ 2s-1 \geq 1 \mid \exists i,j: x_i + s  \equiv_n x_j  \}+1\\
&=2\inf \{ s \geq 1 \mid \exists i,j: x_i + s  \equiv_n x_j  \}.
\end{align*}
\end{proof}

We can now state the corrected conjecture of Abrar as the next proposition by showing that the bound is optimal:

\begin{proposition}
\label{strictstrict}
A non-selfinjective Nakayama algebra with $n \geq 2$ simple modules has its dominant dimension bounded above by $2n-2$ and this bound is optimal for every $n\geq 2$, that is, there exists a non-selfinjective Nakayama algebra with $n$ simple modules and dominant dimension $2n-2$. This Nakayama algebra has $n-1$ injective-projective indecomposable modules.
\end{proposition}
\begin{proof}
The bound was given in Theorem \hyperref[Grenze2]{ \ref*{Grenze2}}.
Using the previous corollary, we take the algebra $C=End_A(A \oplus e_0A/e_0J^{w-1})$, where $A$ is a symmetric Nakayama algebra with $n-1$ simple modules and Loewy length $w$. Then $C$ is a Nakayama algebra with $n$ simple modules and dominant dimension $2n-2$.
\end{proof}

We give another application, showing how to construct algebras of arbitrary dominant dimension larger or equal to two:
\begin{corollary}
Let $w \equiv_n 2$ and $B$ as above.
Then $B$ has dominant dimension \newline $domdim(B)=\inf \{ k \geq 1| \exists x_i , x_j : x_j \equiv_n x_i+k  \}+1.$
\end{corollary}
\begin{proof}
This follows, since in case $w \equiv_n 2$: \newline
$[\frac{k+1}{2}]w-g_k=k+1$ and therefore: \newline
$domdim(B)=\inf \{ k \geq 1| \exists x_i , x_j : x_j +w-1 \equiv_n x_i +[\frac{k+1}{2}]w-g_k  \}+1= \newline
\inf \{ k \geq 1| \exists x_i , x_j : x_j +1 \equiv_n x_i +k+1  \}+1= \newline
\inf \{ k \geq 1| \exists x_i , x_j : x_j \equiv_n x_i +k \}+1.$
\end{proof}
So in this case the dominant dimension is simply equal to one plus the minimal distance of two special points. Like this, one can construct a family of Nakayama algebras with dominant dimension an arbitrary number larger than or equal to two.

\subsection{Gorenstein dimensions of Nakayama algebras which are Morita algebras}
We first recall definitions and standard facts about approximations. Note that by maps we always mean $A$-homomorphisms, when we speak about modules.
\begin{definition}
Let $M$ and $N$ be $A$-modules. 
Recall that a map $g:M \rightarrow N$ is called right minimal in case $gh=g$ implies that $h$ is an isomorphism for any map $h:M \rightarrow M$.
A map $f:M_0 \rightarrow X$, with $M_0 \in$ add($M$), is called a right add($M$)-approximation of $X$ iff the induced map $Hom(N,M_0) \rightarrow Hom(N,X)$ is surjective for every $N \in$ add($M$).
Note that in case $M$ is a generator, such an $f$ must be surjective.
When $f$ is a right minimal homomorphism, we call it a minimal right add($M$)-approximation.
Note that minimal right add($M$)-approximations always exist for finite dimensional algebras.
Given a right minimal map $X$: $f: M_0 \rightarrow X$ with kernel $T$, one obtains the following short exact sequence:
$$0 \rightarrow T \rightarrow M_0 \rightarrow X \rightarrow 0.$$
Applying the functor $Hom(Z,-)$ for some $Z \in add(M)$, one obtains the sequence:
$$0 \rightarrow Hom(Z,T) \rightarrow Hom(Z,M_0) \rightarrow Hom(Z,X) \rightarrow Ext^{1}(Z,T)$$
and thus the induced map $Hom(Z,M_0) \rightarrow Hom(Z,X)$ is surjective in case $Ext^{1}(Z,T)=0$. 
The kernel of such a minimal right add($M$)-approximationl $f$ is denoted by $\Omega_M(X)$. Inductively we define $\Omega_M^{0}(X):=X$ and $\Omega_M^{n}(X):=\Omega_M(\Omega_M^{n-1}(X))$.
The add($M$)-resolution dimension of a module $X$ is a defined as:
$$M\text{-resdim}(X):=\inf \{ n \geq 0 | \Omega_M^{n}(X) \in \text{add}(M) \}.$$
\end{definition}
We use the following Proposition 3.11. from \cite{CheKoe} in order to calculate the Gorenstein dimensions:
\begin{proposition}
\label{CheKoetheorem}
Let $A$ be a finite dimensional algebra and $M$ a CoGen of mod-$A$ and define $B:=End_A(M)$.
Let $B$ have dominant dimension $z+2$, with $z \geq 0$.
Then, for the right injective dimension of $B$ the following holds:
$$\text{injdim}(B)=z+2\ +\ M\text{-resdim}(\tau_{z+1}(M)\oplus D(A)).$$
Here we use the common notation $\tau_{z+1}=\tau\Omega^{z}$, introduced by Iyama (see \cite{Iya}).
\end{proposition}

We note that the Gorenstein symmetry conjecture (which says that the injective dimensions of $A$ and $A^{op}$ are the same) is known to hold for algebras with finite finitistic dimension (see \cite{ARS} page 410, conjecture 13), and thus for Nakayama algebras which are our main examples. Therefore, we will only look at the right injective dimension at such examples.
\ \newline
We now fix our notation as in the previous section: $A$ is a selfinjective Nakyama algebra with $n$ simple modules, Loewy length $w$ and $M=\bigoplus\limits_{i=0}^{n-1}{e_i A} \oplus \bigoplus\limits_{i=1}^{r}{e_{x_i} A / e_{x_i} J^{w-1}}$.
Using the same notation as in the above theorem, we note that $B$ is derived equivalent to $C=End_{A}(A \oplus N)$ (see \cite{HuXi} Corollary 1.3. (2)), with the semisimple module $N=\Omega^{1}(M)=\bigoplus\limits_{i=1}^{r}{e_{x_i}J^{w-1}}$. We also set $W:=A \oplus N$ and we fix that notation for the rest of this section.
\begin{lemma}
The above mentioned derived equivalence between $B$ and $C$ preserves dominant dimension and Gorenstein dimension (We refer to \cite{ChMar}, corollary 4.3. for a more general statement).
\end{lemma}
\begin{proof}
In \cite{HuXi} Corollary 1.2., it is proved that such a kind of derived equivalence preserves dominant dimension and finitistic dimension. If the Gorenstein dimension is finite, it is equal to the finitistic dimension. Since a derived equivalence also preserves the finiteness of Gorenstein dimension, the result follows. 
\end{proof}

We see that we need to know how to calculate minimal right add($W$)-approximations of an arbitrary module in a selfinjective Nakayama algebra. For this we have the following lemma:

\begin{lemma}
Let $e_aJ^y$ be an arbitrary non-projective indecomposable module in the selfinjective Nakayama algebra $A$ and assume that this module is not contained in add($N$). 
\begin{enumerate}
\item If $a \neq x_i$ for all $i=1,...,r$, then the projective cover $e_{a+y}A \rightarrow e_aJ^y \rightarrow 0$ is a minimal right add($W$)-approximation of $e_aJ^y$. \newline
\item If there is an $x_i$ with $a=x_i$, then we have the following short exact sequence:
$$0\rightarrow e_{x_i+y}J^{w-(y+1)} \rightarrow e_{x_i+y}A \oplus e_{x_i}J^{w-1} \rightarrow e_{x_i} J^{y} \rightarrow 0.$$
Here, the map $e_{x_i+y}A \oplus e_{x_i}J^{w-1} \rightarrow e_{x_i} J^{y}$ is the sum of the projective cover of $e_{x_i} J^{y}$ and the socle inclusion of $e_{x_i}J^{w-1}$ in $e_{x_i} J^{y}$.
Then the surjective map in the above short exact sequence is a minimal right add($W$)-approximation.
\end{enumerate}
\end{lemma}
\begin{proof}
\begin{enumerate}
\item The projective cover is clearly minimal. The kernel of the projective cover is $e_{a+y}J^{w-y}$ and we have to show $Ext^{1}(Z,e_{a+y}J^{w-y})=0$ for every $Z \in\ $add($W$).
Since $W$ is a direct sum of simple and projective modules, this simply means that $I_1$ (the first term in a minimal injective coresolution of $e_{a+y}J^{w-y}$) has a socle, which does not lie in add($W$).
But this is true because of $I_1=e_aA$ and our assumption in i). \newline
\item Again, the minimality is obvious. At first we show that the short exact sequence exists. What is left to show is that the kernel is really $e_{x_i+y}J^{w-(y+1)}$.
With
$$e_{x_i}J^y \cong e_{x_i+y}A/e_{x_i+y}J^{w-y}$$
and
$$e_{x_i}J^{w-1} \cong e_{x_i+y}J^{w-y-1}/e_{x_i+y}J^{w-y}$$
we see that the map of interest has up to isomorphism the following form:
$$ f: e_{x_i+y}A \oplus e_{x_i+y}J^{w-y-1}/e_{x_i+y}J^{w-y} \rightarrow e_{x_i+y}A/e_{x_i+y}J^{w-y}.$$
We have $f(w_1,\overline{w_2})=\overline{w_1}+\overline{w_2}$, when $\overline{w}$ denotes the residue class of an element $w$. A basis of the kernel is thus given by the elements
$$\{ ( \phi_{x_i+y,l},0) \mid w-1 \geq l \geq w-y \} \cup \{(\phi_{x_i+y,w-y-1}, -\overline{\phi_{x_i+y,w-y-1})} \},$$
when we denote by $\phi_{c,d}$ the unique path starting at $c$ and having length $d$.

A basis of the socle of the kernel is given by $(\phi_{x_i+y,w-1},0)$ and thus the kernel is isomorphic to $e_{x_i+y}J^{w-(y+1)}$ (by comparing dimension and socle). We now have to show that the induced map $Hom(G,e_{x_i}J^{w-1} \oplus e_{x_i+y}A) \rightarrow Hom(G,e_{x_i}J^y)$ is surjective for every $G \in add(W)$. Note that we can assume that $G$ has no simple summands $S$ which are not isomorphic to $e_{x_i}J^{w-1}$, since we would have $Hom(S,e_{x_i}J^y)=0$ then. With this assumption we get
$$Ext^{1}(G,e_{x_i+y}J^{w-(y+1)}) = 0,\ \text{iff}\ Ext^{1}(e_{x_i}J^{w-1},e_{x_i+y}J^{w-(y+1)}) = 0,$$
and this is true, since the minimal injective presentation of $e_{x_i+y}J^{w-(y+1)}$ is the following:
$$0 \rightarrow e_{x_i+y}J^{w-(y+1)} \rightarrow e_{x_i+y} A \rightarrow e_{x_i-1}A.$$
Then $Ext^{1}(G,e_{x_i+y}J^{w-(y+1)}) = 0$ and thus, the induced map \newline $Hom(G,e_{x_i}J^{w-1} \oplus e_{x_i+y}A) \rightarrow Hom(G,e_{x_i}J^y)$ is surjective. 
\end{enumerate}
\end{proof}

Now we will use this result to calculate the Gorenstein dimensions of gendo-symmetric Nakayama algebras.
We note that for a simple module $S$, $\tau_{z+1}(S)$ is always a simple module, if the dominant dimension of $B$ is even. It is a radical of a projective indecomposable module, if the dominant dimension of $B$ is odd. So, in order to calculate the Gorenstein dimension, it is enough to calculate the minimal right add($W$)-resolutions for modules of the form $(a,w-1)$ and $(a,1)$ for a point $a$.
A diagram of the form
$$\xymatrix@1{  A' \ar@{-}[d]_{1}  \ar@{-}[dr]^{2} \\ B' & C' }$$
means that the kernel of a $W$-approximation of the indecomposable nonprojective module $A'=e_a J^k$ is $B'$, in case $e_aJ^{w-1}$ is not a summand of $W$ (always corresponding to the arrow with a $1$), and the kernel is $C'$ otherwise (always corresponding to an arrow with a $2$).\newline
So for a general module $(a,k)=e_aJ^k$, not in add($W$), the diagram looks as follows in the first step:
$$\xymatrix@1{  (a,k) \ar@{-}[d]_{1}  \ar@{-}[dr]^{2} \\ (a+k,w-k) & (a+k,w-(k+1)) }$$
We also set $B'=stop$, if $B'$ is a summand of $W$. Dots like $\cdots$ indicate that it is clear how the resolution continues from this point on.

\begin{theorem}
\label{gordimgen}
Let $w \equiv_n 1$ (which is equivalent to $A$ being a symmetric Nakayama algebra). Then $B$ has Gorenstein dimension
$$2\sup \{ u_i \mid u_i=\inf \{b \geq 1 \mid \exists j:\ x_i+b \equiv_n x_j  \} \},$$
which is two times the maximal distance between two special points.
\end{theorem}
\begin{proof}
By \hyperref[formgendo]{ \ref*{formgendo}}, $B$ has dominant dimension equal to $2\inf \{ s \geq 1 \mid \exists i,j: x_i + s  \equiv_n x_j  \}$, which is equal to two times the smallest distance of two special points. Denote by $d$ the smallest distance between two special points. Using \hyperref[CheKoetheorem]{ \ref*{CheKoetheorem}} and $\tau \cong \Omega^{2}$ (since $A$ is symmetric), the Gorenstein dimension is equal to $2d$+$W$-resdim($\Omega^{2d}(W)$), with $W=\bigoplus\limits_{i=0}^{n-1}{e_i A} \oplus \bigoplus\limits_{i=1}^{r}{e_{x_i} J^{w-1}}$. Note that $\Omega^{2d}(W)=\bigoplus\limits_{i=1}^{r}{e_{x_i+d} J^{w-1}}$ and so we have to calculate $W$-resdim($\Omega^{2d}(W))$. Since the resolution dimension of a direct sum of modules equals the supremum of the resolution dimensions of the indecomposable summands, it is enough to look at a resolution of a single simple module of the form $(x_j+d,w-1)$:
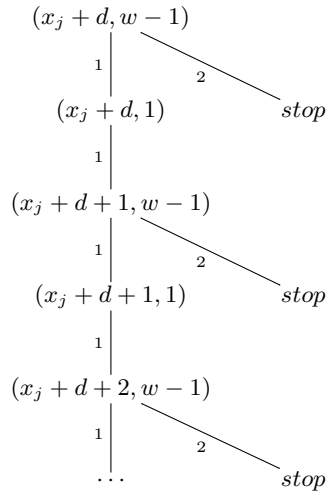
\begin{small}
$$\xymatrix@1{  (x_j+d,w-1) \ar@{-}[d]_{1}  \ar@{-}[dr]_{2} \\ (x_j+d,1) \ar@{-}[d]_{1} & stop\\ (x_j+d+1,w-1) \ar@{-}[d]_{1}  \ar@{-}[dr]_{2} \\ (x_j+d+1,1) \ar@{-}[d]_{1} & stop\\ (x_j+d+2,w-1) \ar@{-}[d]_{1}  \ar@{-}[dr]_{2} \\ \cdots & stop  }$$ 
\end{small}
Considering this diagram, we see now that the resolution finishes exactly when the kernel is of the form $(x_j+d+i,w-1)$ with the smallest $i \geq 0$ such that $e_{x_j+d+i} J^{w-1}$ is a summand of $W$. This takes $2i$ steps. Now the result is clear.
\end{proof}

It follows that the dominant dimension (Gorenstein dimension) of a nonsymmetric gendo-symmetric Nakayama algebra $A$ can be calculated purely graph theoretically: \newline It is two times the minimal (two times the maximal) distance of special points in the quiver of the symmetric Nakayama algebra $eAe$, when $e$ is a primitive idempotent, such that $eA$ is a minimal faithful projective-injective module of $A$. \newline
Combining these results, we get the following geometric characterisation when the dominant dimension equals the Gorenstein dimension for a non-selfinjective gendo-symmetric Nakayama algebra: 
\begin{corollary}
In the situation of the above theorem, injdim$(B)$=domdim$(B)$ iff all the special points in $M$ have the same distance from one another.
\end{corollary}


\begin{thebibliography}{Gus}
\bibitem[Abr]{Abr} Abrar, Muhammad : {\it Dominant dimensions of finite dimensional algebras.} \url{https://elib.uni-stuttgart.de/handle/11682/5085}, PhD thesis, Stuttgart. 
\bibitem[AnFul]{AnFul} Anderson, Frank W.; Fuller, Kent R.: {\it Rings and Categories of Modules.} Graduate Texts in Mathematics, Volume 13, Springer-Verlag, 1992. 
\bibitem[ARS]{ARS} Auslander, Maurice; Reiten, Idun; Smalo, Sverre: {\it Representation Theory of Artin Algebras.}
Cambridge Studies in Advanced Mathematics, Volume 36, Cambridge University Press, 1997.
\bibitem[Ben]{Ben} Benson, David J.: {\it Representations and cohomology I: Basic representation theory of finite groups and associative algebras.} Cambridge Studies in Advanced Mathematics, Volume 30, Cambridge University Press, 1991.
\bibitem[CIM]{CIM} Chan, Aaron; Iyama, Osamu; Marczinzik, Ren\'e : {\it Auslander-Gorenstein algebras from Serre-formal algebras via replication.} In preparation.
\bibitem[ChMar]{ChMar} Chan, Aaron; Marczinzik, Ren\'e : {\it On representation-finite gendo-symmetric biserial algebras.} https://arxiv.org/abs/1607.05965.
\bibitem[Che]{Che} Chen, Xiao-Wu: {\it Gorenstein Homological Algebra of Artin
Algebras.} \url{http://home.ustc.edu.cn/~xwchen/Personal%20Papers/postdoc-Xiao-Wu%20Chen%202010.pdf}, retrieved 18.06.2015. 
\bibitem[CheKoe]{CheKoe} Chen, Hongxing; Koenig, Steffen: {\it Ortho-symmetric modules, Gorenstein algebras and derived equivalences.} electronically published. doi:10.1093/imrn/rnv368.
%\bibitem[CheYe]{CheYe} Chen, Xiao-Wu ; Ye, Yu: {\it Retractions and Gorenstein Homological properties.} Algebras and Representation Theory, Volume 17, pages 713-733, 2014.
\bibitem[FanKoe]{FanKoe} Fang, Ming; Koenig, Steffen: {\it Gendo-symmetric algebras, canonical comultiplication, bar cocomplex and dominant dimension.} electronically published. doi:10.1090/tran/6504.
%\bibitem[FanKoe2]{FanKoe2} Fang, Ming, Koenig, Steffen: {\it Schur functors and dominant dimension}
\bibitem[Farn]{Farn} Farnsteiner, Rolf: {\it Algebras of Finite Global Dimension: Acyclic Quivers.} \url{https://www.math.uni-bielefeld.de/~sek/select/Acyclic.pdf}, retrieved 2015-06-18.
\bibitem[Ful]{Ful} Fuller, Kent R.: {\it Generalized Uniserial Rings and their Kupisch Series.} Math. Zeitschr. Volume 106, pages 248-260, 1968.
\bibitem[Gus]{Gus} Gustafson, William H.: {\it Global dimension in serial rings.} Journal of Algebra, Volume 97, pages 14-16, 1985.
%\bibitem[Hap]{Hap} Happel, Dieter: {\it On Gorenstein Algebras.} Representation Theory of Finite Groups and Finite-Dimensional Algebras
Volume 95 of the series Progress in Mathematics, pp 389-404.
%\bibitem[Hap]{Hap} Happel, Dieter: {\it Homological conjectures in representation theory of finite-dimensional algebras}
%\bibitem[HGK]{HGK} Hazewinkel, Michael ; Gubareni, Nadiya ; Kirichenko, V.V.: {\it Algebras, Rings and Modules} 
\bibitem[HuXi]{HuXi} Hu, Wei; Xi, Changchang: {\it Derived equivalences and stable equivalences of Morita type, I.} Nagoya Mathematical Journal, Volume 200, pages 107-152, 2010.
\bibitem[IgTo]{IgTo} Igusa, Kiyoshi; Todorov, Gordana: {\it On the finitistic global dimension conjecture for Artin algebras.} Journal of Algebra, Volume 320, pages 253-258, 2008.
\bibitem[Iya]{Iya} Iyama, Osamu: {\it Higher-dimensional Auslander-Reiten theory on maximal orthogonal subcategories.} Advances in Mathematics, Volume 210, pages 22-50, 2007.
\bibitem[KerYam]{KerYam} Kerner, Otto; Yamagata, Kunio: {\it Morita algebras.} Journal of Algebra, Volume 382, pages 185-202, 2013.
%\bibitem[Kon]{Kon} Kong, Fan: {\it Generalization of the Correspondence about DTr-selfinjective algebras}
\bibitem[Lam]{Lam} Lam, Tsi-Yuen: {\it A First Course in Noncommutative Rings}. Graduate Texts in Mathematics, Volume 131, Springer-Verlag, 1991.
\bibitem[Mar]{Mar} Marczinzik, Ren\'e: {\it Gendo-symmetric algebras, dominant dimensions and Gorenstein homological algebra.} https://arxiv.org/abs/1608.04212.
\bibitem[Mar2]{Mar2} Marczinzik, Ren\'e: {\it A bocs theoretic characterization of gendo-symmetric algebras.} Journal of Algebra, Volume 470, 15 January 2017, Pages 160-171.
\bibitem[Mue]{Mue} Mueller, Bruno: {\it The classification of algebras by dominant dimension.} Canadian Journal of Mathematics, Volume 20, pages 398-409, 1968.
\bibitem[Nak]{Nak} Nakayama, Tadashi: {\it On algebras with complete homology.} Abhandlungen aus dem Mathematischen Seminar der Universitaet Hamburg, Volume 22, pages 300-307, 1958.
%\bibitem[Rin]{Rin} Ringel, Claus Michael : {\it The Gorenstein projective modules for Nakayama algebras I.} 
%\bibitem[SimSko1]{SimSko1} Simson, Daniel; Skowronski, Andrzej: {\it Elements of the Representation Theory of Associative Algebra 1}
\bibitem[SimSko3]{SimSko3} Simson, Daniel; Skowronski, Andrzej: {\it Elements of the Representation Theory of Associative Algebras, Volume 3: Representation-Infinite Tilted Algebras.}
 London Mathematical Society Student Texts, Volume 72, 2007. 
\bibitem[SkoYam]{SkoYam} Skowronski, Andrzej; Yamagata, Kunio: {\it Frobenius Algebras I: Basic Representation Theory.} EMS Textbooks in Mathematics, 2011.
\bibitem[Ta]{Ta} Tachikawa, Hiroyuki: {\it Quasi-Frobenius Rings and Generalizations: QF-3 and QF-1 Rings (Lecture Notes in Mathematics 351) } Springer; 1973.
\bibitem[Yam]{Yam} Yamagata, Kunio: {\it Frobenius Algebras} in
 {Hazewinkel, M. (editor): Handbook of Algebra, North-Holland, Amsterdam, Volume I, pages 841-887, 1996.}
\bibitem[Yam2]{Yam2} Yamagata, Kunio: {\it Modules with serial Noetherian endomorphism rings.}
 Journal of Algebra, Volume 127, pages 462-469, 1989.
\end{thebibliography}
\end{document}